\documentclass[12pt]{amsart}
\usepackage{amssymb}
\usepackage{amsmath}
\usepackage{mathtools}
\usepackage{amsfonts}
\usepackage{mathrsfs}
\usepackage{psfrag}
\usepackage{verbatim}
\usepackage{color}
\usepackage{enumitem}
\usepackage{bm}
\usepackage[usenames,dvipsnames]{xcolor}
\usepackage{subcaption}
\usepackage{graphicx}
\usepackage{tikz}
\usepackage{color}
\usepackage{pgfplots}
\usepackage{tikz}
\usetikzlibrary{calc}

\numberwithin{equation}{section}

\usepackage[margin=1.2in]{geometry}
\newcommand{\ind}{\boldsymbol{1}}

\usepackage[colorlinks=true]{hyperref}
\hypersetup{urlcolor=blue, citecolor=MidnightBlue}

\DeclareGraphicsExtensions{}
\theoremstyle{plain}
\newtheorem{thm}{Theorem}[section]

\newtheorem{lem}[thm]{Lemma}
\newtheorem{prop}[thm]{Proposition}

\newtheorem{defn}[thm]{Definition}

\newtheorem{thma}{Theorem}

\theoremstyle{remark}
\newtheorem{remark}[thm]{Remark}

\newcommand{\MMM}{\mathcal{M}}

\newcommand{\Sn}{\mathsf{S}_n}
\newcommand{\RR}{\mathbb{R}}

\newcommand{\NN}{\mathbb{N}}
\newcommand{\EE}{\mathbf{E}}
\newcommand{\CCC}{\mathcal{C}}
\newcommand{\JJJ}{\mathcal{J}}
\newcommand{\full}[1]{\widetilde{#1}}
\newcommand{\jp}[1]{\hat{#1}}
\newcommand{\jj}[1]{{#1}}
\newcommand{\jm}[1]{\check{#1}}

\newcommand{\OmH}{\Omega^{\mathsf{H}}}
\newcommand{\dH}{d^{\mathsf{H}}}
\newcommand{\DTV}{d_{\mathsf{TV}}}
\newcommand{\htop}{h_{\mathsf{top}}}
\DeclareMathOperator{\Lip}{Lip}

\begin{document}
\title[Effective uniqueness in the $d$-bar metric]{Effective uniqueness of the measure of maximal entropy in Ornstein's $\bar{d}$-metric}

\author{Vaughn Climenhaga}
\author{Teena Kumari}

\address{Department of Mathematics and Statistics \\ UNC Charlotte, Charlotte, NC, 28223}
\email{vclimenh@charlotte.edu}

\address{Department of Mathematics \\ University of Houston, Houston, TX, 77004}
\email{klnu@cougarnet.uh.edu}

\thanks{This material is based upon work supported by the National Science Foundation under Award No.\ DMS-2154378 and DMS-2453314.}

\date{\today}

\maketitle

\begin{abstract}
Every topologically mixing shift of finite type has a unique measure of maximal entropy. It follows from Ornstein theory that every invariant measure with entropy close to maximal must be close to the MME in the $\bar{d}$-metric. We prove a quantitative version of this, strengthening earlier results of Kadyrov that established effective uniqueness in the Wasserstein metric.
\end{abstract}

\section{Introduction}\label{section : intro}

\subsection{Main results}

Let $\Sigma$ be a mixing shift of finite type (SFT) on a finite alphabet $S$, and let $m$ be its unique measure of maximal entropy (MME), so that $h(m) = \htop(\Sigma)$. By upper semi-continuity of the entropy map $\mu \mapsto h(\mu)$, the following is true for every sequence of shift-invariant Borel probability measures $\mu_n$ on $\Sigma$:
\begin{equation}\label{eqn: continuity of entropy map}
\text{if $h(\mu_n) \to \htop(\Sigma)$, then $\mu_n \to m$ in the weak*-topology.}
\end{equation}
Writing $\mu(\phi) := \int\phi\,d\mu$, the weak*-topology is metrized by the Wasserstein metric
\begin{equation}\label{eqn : W_1}
W_{1}(\mu, \nu)
:= 
\sup \{
|\mu(\phi) - \nu(\phi)|
: \phi \in \Lip^1(\Sigma) \},
\end{equation}
where $\Lip^1(\Sigma)$ denotes the set of $1$-Lipschitz functions on $\Sigma$.
In \cite{kadyrov2015effective}, Kadyrov proved an effective version of \eqref{eqn: continuity of entropy map}, which shows that there exists $C = C(\Sigma)>0$ such that for any invariant measure $\mu$ on $\Sigma$, we have
\begin{equation}\label{eqn: EKP by Kadyrov}
W_{1}(\mu, m) \leq C \sqrt{ \htop(\Sigma) - h(\mu)}.
\end{equation}
We prove the following strengthening of \eqref{eqn: EKP by Kadyrov} in terms of Ornstein's $\bar{d}$-metric on the space of invariant measures, whose definition we recall in \eqref{eqn:d-bar} below.

\begin{thma}\label{thm: Main Theorem}
Let $\Sigma$ be a mixing SFT, and let $m$ be its unique MME. Then there exists a constant $C>0$ (see \eqref{eqn:C}) such that for any shift-invariant measure $\mu$ on $\Sigma$, we have 
\begin{equation}\label{eqn : d_bar inequality in main theorem}
\bar{d}(\mu, m)\leq C \sqrt{ \htop(\Sigma) - h(\mu) }.
\end{equation}
\end{thma}

\begin{remark}
Ornstein introduced the $\bar{d}$-metric as part of his proof that Bernoulli shifts with the same entropy are isomorphic \cite{dO70a,dO74}.
Convergence in $\bar{d}$ is stronger than weak*-convergence: indeed,
\begin{equation}\label{eqn:h-cts}
\text{if $\bar{d}(\mu_n,\mu)\to 0$, then
$\mu_n\xrightarrow{\text{weak*}} \mu$ and $h(\mu_n)\to h(\mu)$.}
\end{equation}
If the measure $\mu$ has the property that the converse of \eqref{eqn:h-cts} holds, then $\mu$ is said to be \emph{finitely determined}. 
The finitely determined measures are exactly those which are isomorphic to Bernoulli shifts \cite{dO70,OW74}; it is well-known that the unique MME of every mixing SFT has this \emph{Bernoulli property}, and is thus finitely determined.  Theorem \ref{thm: Main Theorem} can be viewed as a strengthening of this fact, in the same sense that Kadyrov's result \eqref{eqn: EKP by Kadyrov} strengthens \eqref{eqn: continuity of entropy map}. By the same token, one can view Theorem \ref{thm: Main Theorem} as providing an alternative proof that $m$ has the Bernoulli property.
\end{remark}


Our proof of Theorem \ref{thm: Main Theorem} relies heavily on a result of Marton 
\cite{marton1996bounding} that controls the $\bar{d}$-distance between the finite-dimensional marginals of $\mu$ and $m$ in terms of their informational (Kullback--Liebler) divergence. This extends Pinsker's inequality (see Lemma \ref{lem:Pinsker}), which is also central to Kadyrov's argument.

A different proof of \eqref{eqn: EKP by Kadyrov} was given by R\"uhr and Sarig \cite{ruhr2022effective}, who extended the result to a class of equilibrium states for countable-state Markov shifts (assuming a strong positive recurrence condition). Their argument relies on the fact that the variational principle provides a Legendre transform relationship between the topological pressure function $\phi \mapsto P(\phi)$ and the measure-theoretic entropy $\mu \mapsto h(\mu)$, along with a bound on the second derivative of pressure given by combining the Green--Kubo formula with (uniformly) exponential decay of correlations: in particular, they prove existence of a constant $M>0$ such that given any $\phi \in \Lip^1(\Sigma)$, we have
$\frac {d^2}{dt^2}P(t\phi) \leq M$, and consequently,
\begin{equation}\label{eqn:P-leq}
P(t\phi) \leq \htop(\Sigma) + m(\phi) t + M t^2
\quad\text{for all } t\in \mathbb{R}.
\end{equation}
Since our main result \eqref{eqn : d_bar inequality in main theorem} strengthens \eqref{eqn: EKP by Kadyrov}, it is natural to ask whether there is a corresponding strengthening of \eqref{eqn:P-leq} that extends the inequality to a broader class of $\phi$. We provide such a result in Theorem \ref{thm: pressure estimates} below. The key is a description of $\bar{d}$ analogous to \eqref{eqn : W_1}. Before giving this description, we recall a common definition of the Wasserstein and $\bar{d}$-metrics. 

\begin{defn}\label{def:coupling-joining}
A \emph{coupling} of probability measures $\mu$ and $\nu$ on $\Sigma$ is a probability measure $\theta$ on $\Sigma\times \Sigma$ such that $\pi_*^1 \theta = \mu$ and $\pi_*^2 \theta = \nu$, where $\pi^j \colon \Sigma^2 \to \Sigma$ denotes the projection onto the $j^{\mathrm{th}}$ coordinate. When $\mu$ and $\nu$ are shift-invariant, a \emph{joining} of $\mu$ and $\nu$ is a coupling that is itself shift-invariant. We write $\CCC(\mu,\nu)$ and $\JJJ(\mu,\nu)$ for the collections of all couplings and all joinings, respectively. 
Writing $\rho$ for the metric on $\Sigma$, the
Wasserstein and $\bar{d}$-distances between $\mu$ and $\nu$ are defined by
\begin{align}
\label{eqn:W1}
W_1(\mu,\nu) &:= \inf_{\theta \in \CCC(\mu,\nu)} \int_{\Sigma^2} \rho(x,y) \,d\theta(x,y), \\
\label{eqn:d-bar}
\bar{d}(\mu,\nu) &:= \inf_{\theta\in \JJJ(\mu,\nu)}
\int_{\Sigma^2} \ind_{[x_{0}\neq y_{0}]}(x,y)\, d\theta(x,y).
\end{align}
\end{defn}

The equivalence of the definitions \eqref{eqn : W_1} and \eqref{eqn:W1} for the Wasserstein metric is provided by \emph{Kantorovich--Rubinstein duality} \cite[Theorem 5.10 and Remark 6.5]{Villani2008OptimalTO}. 
We provide an analogous duality result for Ornstein's $\bar{d}$-metric.
Let
\begin{equation}\label{eqn:dnH}
\dH_n(x,y):= \frac{1}{n} \# \left\{ k\in \{ 0,1,2,\dots, n-1\}:\ x_{k}\neq y_{k}  \right\}
\end{equation}
denote the normalized Hamming metric on $S^n$, which gives a pseudo-metric on $\Sigma$. Let $\OmH_n$ denote the set of functions $\phi \colon \Sigma\to \mathbb{R}$ that satisfy the Hamming--Lipschitz bound
$|\phi(x)-\phi(y)|\leq \dH_n(x,y)$ for all $x,y \in \Sigma$, and let $\OmH := \bigcup_{n\in \mathbb{N}} \OmH_n$.

\begin{thma}\label{thm: dual form of d_bar}
Given any two invariant probability measures $\mu$ and $\nu$ on $\Sigma$, we have
\begin{equation}\label{eqn:d-bar-sup}
\bar{d}(\mu, \nu)= \sup\left\{ \left| \mu(\phi)- \nu(\phi) \right| : \phi\in \OmH \right\}.
\end{equation}
\end{thma}

Theorem \ref{thm: dual form of d_bar} is a consequence of a standard description of the $\bar{d}$-metric in terms of the corresponding metric on sets of finite words, together with Kantorovich--Rubinstein duality in that setting; see \S\ref{sec:Ham-Was}.
Using Theorems \ref{thm: Main Theorem} and \ref{thm: dual form of d_bar}, we prove that:

\begin{thma}\label{thm: pressure estimates}
Let $\Sigma$ be a mixing SFT and $m$ its unique MME.
Then there exists a constant $M>0$ such that for all $\phi \in \OmH$, the topological pressure satisfies \eqref{eqn:P-leq}.
\end{thma}

\begin{remark}
Theorem \ref{thm: pressure estimates} provides constraints on the ``shape'' of the pressure function $\phi \mapsto P(\phi)$, and in particular, on its one-dimensional cross-sections $t\mapsto P(t\phi)$ determined by $\phi \in \OmH$, by providing an upper bound on the ``average curvature near $0$''. 
The question of what constraints are placed on one-dimensional cross-sections of the pressure function for various classes of functions has been examined recently by Kucherenko and Quas \cite{KQ22,kucherenko2023asymptotic} 
and by Ma and Pollicott \cite{MP24}.
It is worth comparing \eqref{eqn:P-leq} with \cite[Theorem 4]{kucherenko2023asymptotic},
which provides a lower bound on the average curvature.
\end{remark}

\subsection{Other related literature}\label{sec:Pinsker}

The first proof that a mixing SFT has a unique MME was given by Parry \cite{parry1964intrinsic}, who used conditional measures to show that any MME must be a Markov measure, and that among Markov measures, there is exactly one MME. A different proof of uniqueness, based on showing that any measure that is mutually singular with respect to the MME must have smaller entropy, was given by Adler and Weiss \cite{AW67,AW70,bW73}, and extended by Bowen; see \cite{CT21} for more details.
Ruelle provided a third proof of uniqueness, based on showing (Gateaux) differentiability of the pressure function \cite{dR73}.

A discussion of the literature on effective uniqueness bounds can be found in the work of R\"uhr and Sarig \cite{ruhr2022effective}. They refer to \eqref{eqn: EKP by Kadyrov} as an \emph{Einsiedler--Kadyrov--Polo (EKP) inequality}, because prior to Kadyrov's work in \cite{kadyrov2015effective}, similar bounds were established in the doctoral thesis of Polo \cite{polo2011equidistribution}, who credits Einsiedler with outlining the proof in the case of the doubling map. 
Since then, estimates similar to \eqref{eqn: EKP by Kadyrov} have been obtained in various settings \cite{rR16,sK17,iK17,rR21,ruhr2022effective}.

We remark that our approach, like Kadyrov's, carries a certain echo of Parry's original proof of uniqueness, in the sense that a crucial role is played by studying the conditional measures of $\mu$ and $m$. The approach used by R\"uhr and Sarig, on the other hand, follows in the spirit of Ruelle's proof of uniqueness via differentiability, relying on proving appropriate properties of the topological pressure function.


Finally, we mention that another line of inquiry, with some similarities to the one pursued here, studies how the unique equilibrium state associated to a H\"older continuous potential $\phi$ varies with respect to $\phi$. The dependence can quickly be shown to be weak*-continuous, and 
it has been shown that at least for the full shift,
this can in fact be strengthened to $\bar{d}$-continuity \cite{coelho1998criteria,bhullar2026continuity}.

\subsection{Outline of paper}\label{sec:outline}

In Section \ref{section: Defn}, we give some background material, which includes everything needed to prove Theorem \ref{thm: dual form of d_bar}. Section \ref{section: EKP proof} contains the proof of Theorem \ref{thm: Main Theorem}, the bulk of which consists of translating arguments of Marton from \cite{marton1998measure} into our present setting and filling in details to make the presentation as self-contained as possible. In Section \ref{section: Pressure}, we explain how Theorems \ref{thm: Main Theorem} and \ref{thm: dual form of d_bar} can be used to deduce Theorem \ref{thm: pressure estimates}, following the ideas of R\"uhr and Sarig \cite{ruhr2022effective}.

\section{Background definitions and notation} \label{section: Defn}

\subsection{Topological entropy and pressure}\label{sec:h-and-P}

We recall some of the basic notions from thermodynamic formalism in the context of shift spaces, referring to \cite{walters2000introduction} for further details. Let $S$ be a finite set (the \emph{alphabet}), and $S^\mathbb{N}$ the set of all infinite sequences of symbols from $S$, which we call the \emph{full shift}. Given $x\in S^\mathbb{N}$ and $i,j\in \mathbb{N}$ with $i\leq j$, we write $x_{[i,j]} = x_i x_{i+1} \cdots x_j$ for the subword of $x$ that appears in indices $i$ through $j$. We will also write $x_{(i,j]}$, $x_{[i,j)}$, and $x_{(i,j)}$ to denote the corresponding subwords where one or both of $x_i$ and $x_j$ is omitted.

The \emph{shift map} $\sigma \colon S^{\mathbb{N}} \to S^\mathbb{N}$ is defined by $\sigma(x)_n = x_{n+1}$. It is continuous with respect to the metric
$\rho(x,y) := 2^{-\min \{ n\in \mathbb{N} : x_i \neq y_i \}}$,
where by convention, $\min \emptyset = \infty$ and $2^{-\infty} = 0$. 
A \emph{subshift} (or \emph{shift space}) is a closed subset $\Sigma \subset S^\mathbb{N}$ such that $\sigma(X) \subset \Sigma$.
Given $n\in \mathbb{N}$ and a word $w\in S^n$, the corresponding \emph{cylinder} in $\Sigma$ is
\[
[w] := \{x\in \Sigma : x_{[1,n]} = w\}.
\]
The \emph{language} of $\Sigma$ is $\mathcal{L} := \bigcup_{n\in \mathbb{N}} \mathcal{L}_n$, where $\mathcal{L}_n := \{ w\in S^n : [w] \neq \emptyset \}$.
The \emph{topological entropy} of $\Sigma$ is
\[
\htop(\Sigma) := \lim_{n\to\infty} \frac 1n \log \# \mathcal{L}_n,
\]
where the limit exists by submultiplicativity of $\#\mathcal{L}_n$.
More generally,
given a continuous \emph{potential function} $\phi \colon \Sigma \to \mathbb{R}$, the \emph{topological pressure} of $\phi$ is
\[
P(\phi) := \lim_{n\to\infty} \frac 1n \log \sum_{w\in \mathcal{L}_n}
e^{\sup \{ \Sn \phi(x) : x\in [w] \}},
\quad
\text{where $\Sn\phi(x) := \sum_{j=0}^{n-1} \phi(\sigma^j x)$.}
\]

\subsection{Invariant measures}

Given a compact metric space $X$, we will write $\MMM(X)$ for the set of all Borel probability measures on $X$. In the remainder of the paper, the term ``measure'' will always mean ``Borel probability measure''.
When $X = \Sigma$ is a shift space, we will write $\MMM_\sigma(\Sigma)$ for the set of all shift-invariant measures, that is, those $\mu \in \MMM(\Sigma)$ with the property that $\sigma_* \mu = \mu$.

Given a measure $\mu$ on $\Sigma$ and a natural number $k$, we write $\mu_k$ for the measure induced on $S^k$ by
\begin{equation}\label{eqn:mu-k}
\mu_k(w) := \mu([w])
\quad\text{for all } w\in S^k.
\end{equation}
When $\mu$ is shift-invariant, its \emph{entropy} is
\begin{equation}\label{eqn:h-H}
h(\mu) := \lim_{k\to\infty} \frac 1k \log H_k(\mu),
\quad\text{where }
H_k(\mu) := \sum_{w\in S^k} -\mu(w) \log \mu(w),
\end{equation}
and the limit exists by subadditivity.
(Here we use the convention that $0\log 0 = 0$.)
By the \emph{variational principle}, for every subshift $\Sigma \subset S^\mathbb{N}$ and every continuous potential $\phi \colon \Sigma \to \mathbb{R}$, we have
\[
P(\phi) = \sup_{\mu \in \MMM_\sigma (\Sigma)} \big(h(\mu) + \mu(\phi)\big),
\]
where we use the notation $\mu(\phi) := \int_\Sigma \phi \,d\mu$. An invariant measure achieving the supremum is an \emph{equilibrium state} for $\phi$. When $\phi\equiv 0$, the variational principle becomes
\[
\htop(\Sigma) = \sup_{\mu \in \MMM_\sigma(\Sigma)} h(\mu),
\]
and a measure achieving the supremum is a \emph{measure of maximal entropy} (MME).

We restrict our attention to the following class of shift spaces.
A \emph{transition matrix} $A \colon S\times S \to \{0,1\}$ determines a \emph{shift of finite type} (SFT)
\[
\Sigma := \{ x\in S^\mathbb{N} : A(x_n,x_{n+1}) = 1 \text{ for all } n\in \mathbb{N} \}.
\]
The SFT is \emph{mixing} if there exists $n\in \mathbb{N}$ such that all entries of $A^n$ are positive.

A matrix $T = (T_{ab})_{a,b\in S}$ is \emph{stochastic} if for every $a\in S$, the row vector $(T_{ab})_{b\in S}$ is a probability vector. Given a stochastic matrix $T$ of \emph{transition probabilities}, and a probability vector $q = (q_a)_{a\in S}$, the corresponding \emph{Markov measure} is the probability measure $\mu$ on $S^\mathbb{N}$ defined by
\[
\mu_k(w) := q_{w_1} T_{w_1 w_2} T_{w_2 w_3} \cdots T_{w_{k-1} w_k}
\quad\text{for all } w\in S^k.
\]
The measure $\mu$ is shift-invariant if and only if $q T = q$. If there exists $n\in \mathbb{N}$ such that $(T^n)_{ab}>0$ for all $a,b\in S$, then there exists a unique such $q$, and the resulting Markov measure $\mu$ is mixing.

If $\Sigma$ is the SFT associated to a transition matrix $A$, and if the stochastic matrix $T$ has the probability that $T_{ab}=0$ whenever $A_{ab} = 0$, then $\mu$ is supported on $\Sigma$. 
If $\Sigma$ is mixing, then by the Perron--Frobenius theorem it has a simple eigenvalue $\lambda>0$ such that every other eigenvalue $\beta$ satisfies $|\beta| < \lambda$. The corresponding eigenvectors $\ell,r$ have positive entries, and if they are normalized so that $\sum_{a\in S} \ell_a r_a = 1$, then the probability vector defined by $q_a := \ell_a r_a$ is a stationary vector for the stochastic matrix defined by $T_{ab} := \frac{A_{ab} r_b}{\lambda \ell_a}$. The resulting Markov measure $m$ is the \emph{Parry measure} for $\Sigma$; it is the unique MME \cite{parry1964intrinsic}, and a simple computation shows that there exists $Q\geq 1$ such that $m$ satisfies the \emph{Gibbs property}
\begin{equation}\label{eqn:Gibbs}
Q^{-1} e^{-k\htop(\Sigma)}
\leq m_k(w) \leq Q e^{-k\htop(\Sigma)}
\quad\text{for all } k\in \mathbb{N} \text{ and } w\in S^k.
\end{equation}

\subsection{Metrics on spaces of measures}

We will consider the Wasserstein metric from \eqref{eqn : W_1} and \eqref{eqn:W1} not just in the setting of $\Sigma$ equipped with the metric in \S\ref{sec:h-and-P}, but also on spaces of finite words equipped with both the discrete and Hamming metrics. First, we recall some general theory.

\subsubsection{Wasserstein metric}

Given a compact metric space $(X,\rho)$, and two probability measures $\mu,\nu \in \mathcal{M}(X)$, once again let $\CCC(\mu,\nu) \subset \MMM(X\times X)$ denote the set of all couplings of $\mu$ and $\nu$. Then the Wasserstein distance between $\mu$ and $\nu$ is defined by
\begin{equation}\label{eqn:Was}
W_1(\mu,\nu) := \inf_{\theta \in \CCC(\mu,\nu)} \int_{X^2} \rho(x,y) \,d\theta(x,y),
\end{equation}
just as in \eqref{eqn:W1}. An \emph{optimal coupling} of $\mu$ and $\nu$ is a coupling $\theta \in \CCC(\mu,\nu)$ that achieves this infimum. Since $X$ is compact, every pair of measures admits at least one optimal coupling.

Just as in \eqref{eqn : W_1}, Kantorovich--Rubinstein duality shows that \eqref{eqn:Was} is equivalent to
\begin{equation}\label{eqn:KR}
W_1(\mu,\nu) = \sup \{ |\mu(\phi) - \nu(\phi)| : \phi \in \Lip^1(X) \}.
\end{equation}
In this case when $X$ is finite, it will sometimes be convenient to write the integral of a function $f\colon X\to \RR$ using the probabilistic notation
\[
\EE[f,\mu] := \int_X f \,d\mu = \sum_{x\in X} f(x) \mu(x).
\]
Thus the characterizations of the Wasserstein metric in \eqref{eqn:Was} and \eqref{eqn:KR} can be written as
\begin{equation}\label{eqwn:DTV-2}
\DTV(\mu,\nu) = \inf_{\theta \in \CCC(\mu,\nu)} \EE[\rho,\theta]
= \sup \{ | \EE[\phi,\mu] - \EE[\phi,\nu]| : \phi \in \Lip^1(X) \}.
\end{equation}
We will be particularly interested in the cases when $\rho$ is either the discrete metric or the Hamming metric, and we discuss these next.

\subsubsection{Using the discrete metric}

Let $X$ be a finite set, and let $\rho$ be the discrete metric on $X$. 
The resulting Wasserstein metric on $\MMM(X)$ is the \emph{total variation metric} \cite[Chapter 6]{Villani2008OptimalTO}, which we denote $\DTV$, and which also admits the description
\begin{equation}\label{eqn:DTV}
\DTV(\mu,\nu) := \frac 12 \sum_{a\in X} |\mu(a) - \nu(a)|.
\end{equation}
Given $\mu,\nu \in \MMM(X)$, the \emph{Kullback--Liebler divergence} of $\mu$ with respect to $\nu$ is
\begin{equation}\label{eqn:KL}
D(\mu\|\nu) := \sum_{a\in X} \mu(a) \log \Big( \frac{\mu(a)}{\nu(a)} \Big).
\end{equation}
These two quantities are related by \emph{Pinsker's inequality}:

\begin{lem}\label{lem:Pinsker}
Given any finite set $X$, and any probability measures $\mu,\nu$ on $X$, we have
\begin{equation}\label{eqn:Pinsker}
\DTV(\mu,\nu) \leq \sqrt{\frac 12 D(\mu\|\nu)}.
\end{equation}
\end{lem}

\subsubsection{Using the Hamming metric}\label{sec:Ham-Was}

Now we restrict our attention further to the case of measures on the finite set $S^k$. In addition to the discrete metric, we can also equip $S^k$ with the normalized Hamming metric $\dH_k$ from \eqref{eqn:dnH}, observing that
\begin{equation}\label{eqn:dH-rho}
\dH_k(v,w) = \frac 1k \sum_{j=1}^k \rho(v_j, w_j),
\end{equation}
where $\rho$ is the discrete metric on $S$. We will write $\bar{d}_k$ for the Wasserstein metric on $\MMM(S^k)$ induced by $\dH_k$, so \eqref{eqn:Was} and \eqref{eqn:KR} become
\begin{equation}\label{eqn:bardk}
\bar{d}_k(\mu,\nu)
= \inf_{\theta\in \CCC(\mu,\nu)} 
\EE[\dH_k,\theta]
= \sup \{ |\EE[\phi,\mu] - \EE[\phi,\nu]| : \phi \in \OmH_k \},
\end{equation}
where as in the paragraph following \eqref{eqn:dnH}, $\OmH_k$ is the set of functions $S^k\to \RR$ that are $1$-Lipschitz with respect to $\dH_k$.

The connection between this family of metrics and the $\bar{d}$-metric on $\MMM_\sigma(S^\NN)$ defined in \eqref{eqn:d-bar} is provided by the following result, whose proof can be found in \cite[Lemma 15.24(4) and Theorem 15.25]{glasner2003ergodic}: given any shift-invariant measures $\mu,\nu \in \MMM_\sigma(S^\NN)$, we have
\begin{equation}\label{eqn:dbar-lim-sup}
\bar{d}(\mu,\nu)
= \lim_{k\to\infty} \bar{d}_k(\mu_k,\nu_k)
= \sup_{k\in \NN} \bar{d}_k(\mu_k,\nu_k).
\end{equation}
Combining \eqref{eqn:bardk} and \eqref{eqn:dbar-lim-sup} establishes Theorem \ref{thm: dual form of d_bar}.

The bulk of our main proof will be devoted to establishing a relationship between $\DTV$ and $\bar{d}_k$ (see Proposition \ref{prop:dTV}), so that Pinsker's inequality can be used to bound the $\bar{d}$-metric.
The following result plays no formal role in our later arguments, but is instructive in illustrating the kind of relationship we hope to obtain.

\begin{lem}\label{lem:d-bar-TV}
Consider two probability measures $\nu^1$ and $\nu^2$ on $S^k$ that are product measures in the sense that there exist probability measures $\nu^i_j$ on $S$ such that
\begin{equation}\label{eqn:product}
\nu^i(w) = \prod_{j=1}^k \nu^i_j(w_j)
\quad\text{for all } w\in S^k, \  i\in\{1,2\}.
\end{equation}
Then the $\bar{d}$-metric and the total variation metric are related by
\begin{equation}\label{eqn:d-bar-TV}
\bar{d}_k(\nu^1, \nu^2) = \frac 1k \sum_{j=1}^k \DTV(\nu^1_j, \nu^2_j).
\end{equation}
\end{lem}
\begin{proof}
Given any coupling $\theta \in \CCC(\nu^1,\nu^2)$ and any $1\leq j\leq k$, the measure $\theta_j$ on $S^2$ defined by
\[
\theta_j(a,b) := 
\theta(\{(x,y) \in (S^k)^2 : x_j = a \text{ and } y_j = b\})
\] 
is a coupling of $\nu^1_j$ and $\nu^2_j$, and 
using \eqref{eqn:dH-rho}, 
we have
\begin{equation}\label{eqn:EE}
\begin{aligned}
\mathbf{E}[\dH_k, \theta]
&= \sum_{(x,y) \in (S^k)^2} \dH_k(x,y) \theta(x,y)
= \frac 1k \sum_{j=1}^k \sum_{(x,y) \in (S^k)^2}  \rho(x_j, y_j) \theta(x,y) \\
&= \frac 1k \sum_{j=1}^k \sum_{(a,b) \in S^2} \rho(a,b) \theta_j(a,b)
= \frac 1k \sum_{j=1}^k \mathbf{E} [\rho, \theta_j].
\end{aligned}
\end{equation}
Since $\mathbf{E}[\rho,\theta_j] \geq \DTV(\nu_j^1, \nu_j^2)$, we can take an infimum over all $\theta \in \CCC(\nu^1, \nu^2)$ and obtain
\[
\bar{d}_k(\nu^1, \nu^2) \geq \frac 1k \sum_{j=1}^k \DTV(\nu^1_j, \nu^2_j).
\]
To prove \eqref{eqn:d-bar-TV}, it remains to show the reverse inequality. To this end, consider any $k$-tuple of couplings $\theta_j \in \CCC(\nu_j^1,\nu_j^2)$ for $1\leq j\leq k$; then the probability measure $\theta$ defined on $(S^k)^2$ by
\[
\theta(x,y) := \prod_{j=1}^k \theta_j(x_j,y_j)
\]
is a coupling of $\nu^1$ and $\nu^2$, and using \eqref{eqn:EE}, we have
\[
\frac 1k \sum_{j=1}^k \mathbf{E} [\rho, \theta_j]
= \mathbf{E}[\dH_k, \theta] \geq \bar{d}_k(\nu^1,\nu^2).
\]
Taking an infimum over all $(\theta_j)_j$ completes the proof.
\end{proof}

\section{Effective uniqueness using informational divergence \label{section: EKP proof}}


This section is devoted to the proof of the effective uniqueness result in Theorem \ref{thm: Main Theorem}. 

\subsection{Overview}

We start by showing that the entropy gap on the right-hand side of \eqref{eqn : d_bar inequality in main theorem} can be expressed as the rate of growth of the Kullback--Liebler divergence of $\mu$ with respect to the MME.
Recall that $\mu_n$ denotes the measure that $\mu$ induces on $S^n$ by \eqref{eqn:mu-k}, and $D$ is defined by \eqref{eqn:KL}.

\begin{lem}\label{lem: KL-divergence in terms of entropy}
Let $\Sigma$ be a mixing SFT, and let $m$ be its unique MME (the Parry measure). Then for any $\sigma$-invariant measure $\mu$ on $\Sigma$, we have
\begin{equation}\label{eqn:h-gap}
\lim_{n\to \infty} \frac{1}{n} D\left( \mu_n \| m_n \right) = \htop(\Sigma) - h(\mu).
\end{equation}
\end{lem}

\begin{proof}
Using \eqref{eqn:KL}, we have the following for each $n\in \NN$:
\begin{equation}\label{eqn:Dn}
\begin{aligned}
D\left( \mu_n \| m_n \right)&= \sum_{w\in S^{n}} \mu_n(w)\log\Big( \frac{\mu_n(w)}{m_n(w)} \Big)\\
&= \sum_{w\in S^{n}} -\mu_n(w)\log m_n(w) - \sum_{w\in S^{n}}-\mu_n(w)\log \mu_n(w).
\end{aligned}
\end{equation}
Using the Gibbs property in \eqref{eqn:Gibbs}, we have the following for every $w\in S^n$:
\[
\log m_n(w) = -n\htop(\Sigma) \pm \log Q,
\]
where we use the notation $A=B\pm C$ to mean $B-C\leq A \leq B+C$.
Returning to \eqref{eqn:Dn} and recalling the definition of $H_n(\mu)$ in \eqref{eqn:h-H}, we have
\begin{align*}
D \left( \mu_n \| m_n \right) &= \sum_{w\in S^{n}} -\mu_n(w)(-n\htop(\Sigma) \pm \log Q) - H_{n}(\mu)\\
 &= \pm \log(Q)+ n\htop(\Sigma) - H_{n}(\mu).
\end{align*}
Dividing both sides by $n$ and sending $n\to\infty$ gives \eqref{eqn:h-gap}.
\end{proof}

Lemma \ref{lem: KL-divergence in terms of entropy} and \eqref{eqn:dbar-lim-sup} will allow us to deduce Theorem \ref{thm: Main Theorem} from the following analogous result for the finite-dimensional marginals.

\begin{thm}\label{thm: d_bar distance and KL-divergence}
Let $m$ be a mixing Markov measure on $S^\NN$, and let $T$ be the corresponding stochastic matrix of transition probabilities, so that there exists $k\in \NN$ such that every entry of $T^k$ is positive. Let $q>0$ be such that $(T^k)_{ab}\geq q$ for all $a,b\in S$. Then for every $n\in \NN$, and any $\mu \in \MMM(S^\NN)$, we have
\begin{equation}\label{eqn:Marton}
\bar{d}_{nk}(\mu_{nk},m_{nk})
\leq \frac{k}{q} \sqrt{\frac 2n D_{nk}(\mu_{nk} \| m_{nk})}.
\end{equation}
\end{thm}

The bound in \eqref{eqn:Marton} is a version of a result proved by Marton \cite[Proposition 3]{marton1996bounding}.
We devote \S\ref{sec:TV-bound}--\S\ref{sec:back-to-d-bar} below to its proof, in which we follow Marton's argument, but translate it from her more probabilistic language into our present terminology, and organize it in a way that highlights the key elements on which we rely. The heart of the argument is Proposition \ref{prop:dTV}, which can be viewed as an extension of Lemma \ref{lem:d-bar-TV} to a broader class of measures $m$.

Once Theorem \ref{thm: d_bar distance and KL-divergence} is proved, Theorem \ref{thm: Main Theorem} can be deduced as follows: writing \eqref{eqn:Marton} as
\[
\bar{d}_{nk}(\mu_{nk},m_{nk})
\leq \frac{\sqrt{2} k^{3/2}}{q} \sqrt{\frac 1{nk} D_{nk}(\mu_{nk} \| m_{nk})}
\]
and sending $n\to\infty$, we see that the left-hand side converges to $\bar{d}(\mu,m)$ by \eqref{eqn:dbar-lim-sup}, while the divergence rate inside the square root converges to $\htop(\Sigma)-h(\mu)$ by Lemma \ref{lem: KL-divergence in terms of entropy}. This establishes \eqref{eqn : d_bar inequality in main theorem} with
\begin{equation}\label{eqn:C}
C := \frac{\sqrt{2} k^{3/2}}{q},
\end{equation}
completing the proof of Theorem \ref{thm: Main Theorem}, modulo Theorem \ref{thm: d_bar distance and KL-divergence}.

\subsection{A bound using total variation distance}\label{sec:TV-bound}

In the remainder of \S\ref{section: EKP proof}, we fix $k\in \NN$ as in Theorem \ref{thm: d_bar distance and KL-divergence}.
Given a measure $\mu$ on $S^\NN$, we will use the following notation for the conditional measure on $S^k$ induced by the first $\ell$ entries of $x$: given $\ell,k\in \mathbb{N}$ and $x\in S^\mathbb{N}$ or $x\in S^n$ with $n\geq \ell + k$, we write
\[
\mu_k^{x,\ell}(w)
:= \frac{\mu(x_{[1,\ell]} w)}{\mu(x_{[1,\ell]})}.
\]
When $\ell=0$, this becomes the measure $\mu_k$ on $S^n$ defined in \eqref{eqn:mu-k}.
Given $y\in S^\mathbb{N}$, or $y\in S^n$ for some $n \geq k$, we will also write
\[
\mu_k(y) := \mu_k(y_{[1,k]})
\quad\text{and}\quad
\mu_k^{x,\ell}(y) := \mu_k^{x,\ell}(y_{[1,k]}).
\]
We will use the following consequence of this conditional measure notation: given $k,n\in \mathbb{N}$ and $x\in S^{kn}$, we have
\begin{equation}\label{eqn:mun-cond}
\mu_{nk}(x) = \prod_{j=0}^{n-1} \mu_k^{x,jk}(x^{(j)}),
\end{equation}
where $x^{(j)} := x_{(jk, (j+1)k]}$ is the word of length $k$ that appears in $x$ immediately following position $jk$, so that
\[
x_{[1,nk]} = x^{(0)} x^{(2)} \cdots x^{(n-1)}.
\]
The following result can be viewed as an analogue of Lemma \ref{lem:d-bar-TV}, with \eqref{eqn:mun-cond} serving as the replacement for the restrictive ``product measure'' property required in \eqref{eqn:product}.

\begin{prop}\label{prop:dTV}
Let $m,T,k,q$ be as in Theorem \ref{thm: d_bar distance and KL-divergence},
and let $K := 1 + \frac{k}{q}$.
 Then for every $n\in \NN$ and $\mu \in \MMM(S^\NN)$, we have
\begin{equation}\label{eqn:dTV}
\bar{d}_{nk}(\mu_{nk},m_{nk})
\leq \frac Kn \sum_{j=0}^{n-1}
\sum_{y \in S^{jk}}
\mu(y)
\DTV(\mu_k^{y,jk},m_k^{y,jk}).
\end{equation}
\end{prop}

We prove Proposition \ref{prop:dTV} in \S\S\ref{sec:candidate}--\ref{sec:pf-dTV}. Then in \S\ref{sec:back-to-d-bar}, we use 
\eqref{eqn:dTV} together with Pinsker's inequality and properties of Kullback--Liebler divergence to complete the proof of Theorem \ref{thm: d_bar distance and KL-divergence}.

\subsection{A candidate coupling}\label{sec:candidate}

In order to prove Proposition \ref{prop:dTV}, we will define, for each $n\in \mathbb{N}$, a coupling $\theta \in \CCC(\mu_{nk}, m_{nk})$ such that $\mathbf{E}[\bar{d}_{nk}, \theta]$ is sufficiently small. The construction of $\theta$ will be given in this section, and appropriate bounds on $\mathbf{E}[\bar{d}_{nk}, \theta]$ will be proved in \S\ref{sec:pf-dTV}.

Our arguments will use couplings of three measures, not just of two: given measures $\nu^1,\nu^2,\nu^3$ on a space $X$, we let $\CCC(\nu^1,\nu^2,\nu^3)$ denote the set of all $\zeta \in \MMM(X^3)$ with the property that $\pi^i_* \zeta = \nu^i$ for each $i\in \{1,2,3\}$, where $\pi^i \colon X^3 \to X$ denotes projection to the $i^{\mathrm{th}}$ coordinate.
We will also write $\pi^{12} \colon X^3 \to X^2$ for the projection onto the first two coordinates, and similarly for $\pi^{13}$ and $\pi^{23}$. Observe that if $\zeta \in \CCC(\nu^1,\nu^2,\nu^3)$, then $\pi^{ij}_* \zeta$ is a coupling of $\nu^i$ and $\nu^j$.

\subsubsection{The plan}

To produce the coupling $\theta$, we will also consider an auxiliary measure $\nu$ on $S^{nk}$ that is given (roughly speaking) as follows:
\begin{itemize}
\item if we partition $\{1,2,\dots, nk\}$ into $n$ blocks of $k$ consecutive indices, then each block is independent from the others w.r.t.\ $\nu$;
\item the symbols within a given block are distributed according to the transition probabilities from the matrix $T$ (which determines the Markov measure $m$), with the distribution at the start of the $j^{\mathrm{th}}$ block being given by the distribution of the $(j-1)k^{\mathrm{th}}$ symbol w.r.t.\ $\mu$ (the measure to be compared to $m$).
\end{itemize}
To construct $\nu$ precisely, we will start with optimal couplings $\Gamma_j^y \in \CCC(\mu_k^{y,jk}, m_k^{y,jk})$ on the $j^{\mathrm{th}}$ block of $k$ symbols, then take the product of these (as in \eqref{eqn:mun-cond}) to produce a measure $\gamma \in \MMM((S^{nk})^2)$, whose projection to the second coordinate will be $\nu$. This is the meaning of the right-hand edges of the triangles in Figure \ref{fig:couplings}, which also illustrates the following steps:
\begin{itemize}
\item given optimal couplings $Y_j^{x,y} \in \CCC(m_k^{x,jk}, m_k^{y,jk})$, we take the ``relative product'' of $Y_j^{x,y}$ and $\Gamma_j^y$ to obtain couplings $Z_j^{x,y} \in \CCC(m_k^{x,jk}, \mu_k^{y,jk}, m_k^{y,jk})$;
\item the product of these couplings over $j \in \{1,\dots, n\}$ gives $\zeta \in \CCC(m_{nk},\mu_{nk},\nu)$;
\item the candidate coupling $\theta \in \CCC(m_{nk},\mu_{nk})$
is the projection of $\zeta$ to the first two coordinates.
\end{itemize}
As Figure \ref{fig:couplings} suggests, throughout \S\ref{sec:candidate} we will consistently use the notation $(u,v,w)$ for an element of $(S^k)^3$, and $(x,y,z)$ for an element of $(S^{nk})^3$. (In \S\ref{sec:pf-dTV} and \S\ref{sec:back-to-d-bar}, we will use $(x,y,z)$ to denote an element of $(S^{jk})^3$ for some $0\leq j < n$.)

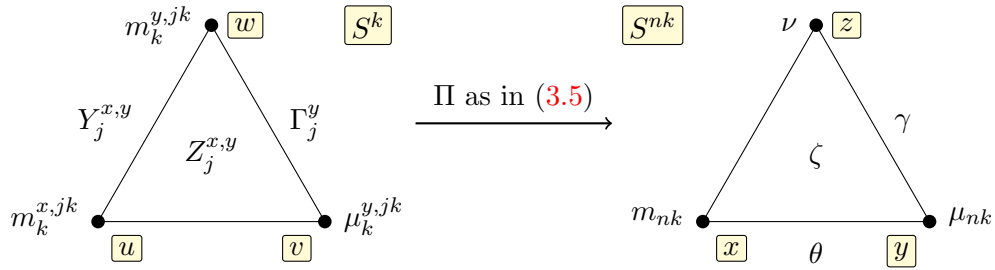
\begin{figure}[htbp]
\begin{tikzpicture}[scale=1,
    every node/.style={font=\small},
    vertex/.style={circle,fill=black,inner sep=1.8pt},
    innerlabel/.style={
        draw,
        fill=yellow!20,
        rectangle,
        rounded corners=1pt,
        inner sep=2.5pt
    }
]

\coordinate (A) at (0,0);
\coordinate (B) at (3,0);
\coordinate (C) at (1.5,{1.5*sqrt(3)});

\coordinate (D) at (8,0);
\coordinate (E) at (11,0);
\coordinate (F) at (9.5,{1.5*sqrt(3)});

\draw (A)--(B)--(C)--cycle;
\draw (D)--(E)--(F)--cycle;

\foreach \P in {A,B,C,D,E,F}
    \node[vertex] at (\P) {};

\node[left=3pt,yshift=2pt]  at (A) {$m_k^{x,jk}$};
\node[right=3pt,yshift=2pt] at (B) {$\mu_k^{y,jk}$};
\node[left=3pt]             at (C) {$m_k^{y,jk}$};

\node[innerlabel,below right=5pt] at (A) {$u$};
\node[innerlabel,below left=5pt]  at (B) {$v$};
\node[innerlabel,right=6pt]       at (C) {$w$};

\node[left=3pt,yshift=2pt]  at (D) {$m_{nk}$};
\node[right=3pt,yshift=2pt] at (E) {$\mu_{nk}$};
\node[left=3pt]             at (F) {$\nu$};

\node[innerlabel,below right=5pt] at (D) {$x$};
\node[innerlabel,below left=5pt]  at (E) {$y$};
\node[innerlabel,right=6pt]       at (F) {$z$};

\coordinate (Lcenter) at (1.5,{0.5*sqrt(3)});
\coordinate (Rcenter) at (9.5,{0.5*sqrt(3)});

\node[left=4pt]  at ($(A)!0.5!(C)$) {$Y_j^{x,y}$};
\node[right=4pt] at ($(B)!0.5!(C)$) {$\Gamma_j^y$};
\node at (Lcenter) {$Z_j^{x,y}$};

\node[below=4pt] at ($(D)!0.5!(E)$) {$\theta$};
\node[right=4pt] at ($(E)!0.5!(F)$) {$\gamma$};
\node at (Rcenter) {$\zeta$};

\draw[->,thick] (4.2,1.3) -- (6.8,1.3)
    node[midway,above] {$\Pi$ as in \eqref{eqn:mun-cond}};

\node[innerlabel, right=50pt] at (C) {$S^k$};
\node[innerlabel, left=50pt] at (F) {$S^{nk}$};
\end{tikzpicture}
\caption{Various couplings involved in the construction.}
\label{fig:couplings}
\end{figure}

\subsubsection{Coupling to an auxiliary measure} 

As described above, we first construct $\gamma$ via its conditional measures on $k$-blocks, and then obtain $\nu$ as the projection of $\gamma$.

Given $y\in S^{nk}$ and $0\leq j < n$, let $\Gamma_j^y$ be a probability measure on $S^k\times S^k$ that is the optimal coupling of the conditional measures $\mu_k^{y,jk}$ and $m_k^{y,jk}$ with respect to the discrete metric on $S^k$; in particular, we have
\begin{equation}\label{eqn:delta-Gamma}
\Gamma_j^y (\{ (v,w) \in S^k\times S^k : v\neq w \})
= \DTV(\mu_k^{y,jk}, m_k^{y,jk}).
\end{equation}
Given $z\in S^{nk}$ and $0\leq j < n$, let $z^{(j)} := z_{(jk, (j+1)k]}$.
Now define a measure $\gamma$ on $S^{nk} \times S^{nk}$ by the property that its conditional measures on each block of $k$ consecutive symbols are given by $\Gamma_j^y$: that is, 
\begin{equation}\label{eqn:gamma}
\gamma(y,z) := \prod_{j=0}^{n-1}
\Gamma_j^y(y^{(j)}, z^{(j)}).
\end{equation}
Since $\Gamma_j^y \in \CCC(\mu_k^{y,jk},m_k^{y,jk})$, we have $\pi^1_* \Gamma_j^y = \mu_k^{y,jk}$, and consequently, given any $y\in S^{nk}$, we have
\begin{align*}
\pi^1_* \gamma(y) &= \sum_{z\in S^{nk}} \gamma(y,z)
= \sum_{z\in S^{nk}} \prod_{j=0}^{n-1}
\Gamma_j^y(y^{(j)},z^{(j)}) \\
&= \prod_{j=0}^{n-1} \sum_{w \in S^k}
\Gamma_j^y(y^{(j)}, w)
= \prod_{j=0}^{n-1} \pi^1_* \Gamma_j^y(y^{(j)})
= \prod_{j=0}^{n-1} \mu_k^{y,jk}(y^{(j)})
= \mu_{nk}(y).
\end{align*}
Writing $\nu := \pi^2_* \gamma$, we see that $\gamma \in \CCC(\mu_{nk}, \nu)$.

\subsubsection{The grand coupling}

To define the coupling $\zeta \in \CCC(m_{nk}, \mu_{nk}, \nu)$, we once again use conditional measures corresponding to an optimal coupling. This time, given $x,y\in S^{nk}$ and $0\leq j < n$, we let $Y_j^{x,y}$ be the measure on $S^k \times S^k$ that is the optimal coupling of the conditional measures $m_k^{x,jk}$ and $m_k^{y,jk}$ with respect to the normalized Hamming metric $\dH_k$ on $S^k$; in particular, we have
\begin{equation}\label{eqn:Y-D}
\mathbf{E} [ \dH_k, Y_j^{x,y}]
= \sum_{u,w\in S^k} \dH_k(u,w) Y_j^{x,y}(u,w)
= \bar{d}_k( m_k^{x,jk}, m_k^{y,jk} ).
\end{equation}
This optimal coupling has the following property:
\begin{equation}\label{eqn:stay-coupled}
\text{if $\dH_k(u,w) < 1$ and $Y_j^{x,y}(u,w) > 0$, then $u_k = w_k$.}
\end{equation}
Roughly speaking, this is because if the coupling were to give positive weight to any pair $(v,w)$ with $\dH_k(v,w) < 1$ and $v_k \neq w_k$, then we could produce another coupling for which the sum in \eqref{eqn:Y-D} takes a strictly smaller value by moving the corresponding weight to pairs of words that remain identical after they first match.

\begin{lem}\label{lem:Zjxy}
For each $x,y\in S^{nk}$ and $j\in \mathbb{N}$, the measure $Z_j^{x,y}$ on $(S^k)^3$ defined by 
\[
Z_j^{x,y}(u,v,w) =
\frac{Y_j^{x,y}(u,w) \Gamma_j^y(v,w)}{m_k^{y,jk}(w)}
\]
has the property that $\pi^{13}_* Z_j^{x,y} = Y_j^{x,y}$ and $\pi^{23}_* Z_j^{x,y} = \Gamma_j^y$, where $\pi^{13}(u,v,w) = (u,w)$ and $\pi^{23}(u,v,w) = (v,w)$. Consequently,  $Z_j^{x,y} \in \CCC(m_k^{x,jk},\mu_k^{y,jk},m_k^{y,jk})$.
\end{lem}
\begin{proof}
It suffices to observe that $\sum_v \Gamma_j^y(v,w) = m_k^{y,jk}(w) = \sum_u Y_j^{x,y}(u,w)$.
\end{proof}

As in the definition of $\gamma$ in \eqref{eqn:gamma}, we now define a measure $\zeta$ on $S^{nk} \times S^{nk} \times S^{nk}$ such that its conditional measures on blocks of length $k$ are given by $Z_j^{x,y}$:
\begin{equation}\label{eqn:zeta}
\zeta(x,y,z) := \prod_{j=0}^{n-1}
Z_j^{x,y}(x^{(j)},y^{(j)},z^{(j)}).
\end{equation}
Using Lemma \ref{lem:Zjxy}, the same argument as in the paragraph following \eqref{eqn:gamma} shows that $\zeta \in \CCC(m_{nk}, \mu_{nk},\nu)$, and thus the measure $\theta := \pi^{12}_* \zeta$ on $S^{nk} \times S^{nk}$, which is given by
\begin{equation}\label{eqn:theta}
\theta(x,y) = \sum_{z\in S^{nk}} \zeta(x,y,z),
\end{equation}
lies in $\CCC(m_{nk}, \mu_{nk})$.

\subsection{Proof of Proposition \ref{prop:dTV}}\label{sec:pf-dTV}

Since the measure $\theta$ on $(S^{nk})^2$ defined in \eqref{eqn:theta} is a coupling of $m_{nk}$ and $\mu_{nk}$, we have
\begin{equation}\label{eqn:Delta-def}
\bar{d}_{nk}(m_{nk},\mu_{nk}) \leq
\sum_{(x,y)\in (S^{nk})^2} \dH_{nk}(x,y) \theta(x,y)
= \mathbf{E}[\dH_{nk}, \theta]
=: \Delta.
\end{equation}
Thus in order to prove Proposition \ref{prop:dTV}, it suffices to produce an appropriate upper bound for $\Delta$. In the estimates that follow, we will use the following notation: 
we will use the symbols $\full{x}, \full{y}, \full{z}$ to represent words in $S^{nk}$, and will carry out computations involving some fixed $j$, writing
\[
\jm{x}, \jm{y}, \jm{z} 
 \in S^{(j-1)k},
\qquad
\jj{x}, \jj{y}, \jj{z} 
\in S^{jk},
\qquad
\jp{x}, \jp{y}, \jp{z} 
\in S^{(j+1)k}.
\]
We will also use the following notation, which is analogous to \eqref{eqn:mu-k}:
\[
\zeta_{jk}(\jj{x},\jj{y},\jj{z})
:= \zeta( \{ (\full{x},\full{y},\full{z}) \in (S^{nk})^3 : \full{x}_{[1,jk]} = \jj{x},\ \full{y}_{[1,jk]} = \jj{y}, \text{ and } \full{z}_{[1,jk]} = \jj{z} \} ),
\]
and similarly for $\zeta_{(j+1)k}(\jp{x},\jp{y},\jp{z})$ and $\zeta_{(j-1)k}(\jm{x},\jm{y},\jm{z})$.
Now we use the fact that
\[
\dH_{nk}(\full{x},\full{y}) = \frac 1n \sum_{j=0}^{n-1} \dH_k(\full{x}^{(j)}, \full{y}^{(j)}),
\]
along with the fact that $\full{x}^{(j)} = \jp{x}^{(j)}$
and $\full{y}^{(j)} = \jp{y}^{(j)}$, to write
\begin{equation}\label{eqn:Delta}
\begin{aligned}
\Delta
&=  \frac 1n \sum_{j=0}^{n-1} \sum_{(\full{x},\full{y},\full{z}) \in (S^{nk})^3} \dH_k(\full{x}^{(j)}, \full{y}^{(j)}) \zeta(\full{x},\full{y},\full{z}) \\
&= \frac 1n \sum_{j=0}^{n-1} \sum_{(\jp{x},\jp{y},\jp{z}) \in (S^{(j+1)k})^3}
\dH_k(\jp{x}^{(j)}, \jp{y}^{(j)})
\zeta_{(j+1)k}(\jp{x},\jp{y},\jp{z}). 
\end{aligned}
\end{equation}
Observe that we can write $\jp{x} = \jj{x} u$, where $u = \jp{x}^{(j)}$, and similarly for $\jp{y}$ and $\jp{z}$. Given $(u,v,w) \in (S^k)^3$, we have
\[
\zeta_{(j+1)k}(\jj{x} u, \jj{y} v, \jj{z} w)
= \zeta_{jk}(\jj{x}, \jj{y}, \jj{z}) Z_j^{\jj{x},\jj{y}}(u,v,w),
\]
so writing $r(u,v,w) := \dH_k(u,v)$, \eqref{eqn:Delta} gives
\begin{equation}\label{eqn:Delta-again}
\begin{aligned}
\Delta &= \frac 1n \sum_{j=0}^{n-1}
\sum_{(\jj{x},\jj{y},\jj{z}) \in (S^{jk})^3}
\sum_{(u,v,w) \in (S^k)^3}
\dH_k(u,v) \zeta_{jk}(\jj{x}, \jj{y}, \jj{z}) Z_j^{\jj{x},\jj{y}}(u,v,w) \\
&= \frac 1n \sum_{j=0}^{n-1}
\sum_{(\jj{x},\jj{y},\jj{z}) \in (S^{jk})^3}
\EE[ r, Z_j^{\jj{x},\jj{y}} ] \zeta_{jk} (\jj{x},\jj{y},\jj{z}).
\end{aligned}
\end{equation}
Now consider the functions $f,g \colon (S^k)^3 \to [0,1]$ defined by
\[
f(u,v,w) =
\begin{cases} 1 &\text{if } v\neq w \\ 0 &\text{if } v=w \end{cases}
\qquad\text{and}\qquad
g(u,v,w) =
\begin{cases} 0 &\text{if } v\neq w \\ \dH_k(u,w) &\text{if } v=w \end{cases}
\]
and observe that $r \leq f + g$, so that
\begin{equation}\label{eqn:EEE}
\mathbf{E}[r, Z_j^{\jj{x},\jj{y}}]
\leq \mathbf{E}[f, Z_j^{\jj{x},\jj{y}}] + \mathbf{E}[g, Z_j^{\jj{x},\jj{y}}].
\end{equation}
Recalling \eqref{eqn:delta-Gamma} and using the fact from Lemma \ref{lem:Zjxy} that $\pi^{23}_* Z_j^{\jj{x},\jj{y}} = \Gamma_j^{\jj{y}}$, we see that
\begin{equation}\label{eqn:Ef}
\mathbf{E}[f, Z_j^{\jj{x},\jj{y}}] = Z_j^{\jj{x},\jj{y}}( \{ (u,v,w) \in (S^k)^3 : v\neq w \})
= \DTV(\mu_k^{\jj{y},jk},m_k^{\jj{y},jk}).
\end{equation}
For the last term in \eqref{eqn:EEE}, we first note that
since $g(u,v,w) \leq \dH_k(u,w)$, we have
\begin{equation}\label{eqn:Eg}
\mathbf{E}[g, Z_j^{\jj{x},\jj{y}}]
\leq \sum_{(u,v,w) \in (S^k)^3} \dH_k(u,w) Z_j^{\jj{x},\jj{y}}(u,v,w)
= \mathbf{E}[\dH_k, Y_j^{\jj{x},\jj{y}}].
\end{equation}
so that \eqref{eqn:EEE} gives
\begin{equation}\label{eqn:EDE}
\mathbf{E}[r, Z_j^{\jj{x},\jj{y}}]
\leq \DTV(\mu_k^{\jj{y},jk},m_k^{\jj{y},jk})
+ \mathbf{E}[\dH_k, Y_j^{\jj{x},\jj{y}}].
\end{equation}
Combining \eqref{eqn:Delta} and \eqref{eqn:EDE}, we have
\begin{equation}\label{eqn:Delta-1}
\Delta \leq \frac 1n \sum_{j=0}^{n-1} \sum_{(\jj{x},\jj{y},\jj{z}) \in (S^{jk})^3}
\zeta(\jj{x},\jj{y},\jj{z}) 
\big( \DTV(\mu_k^{\jj{y},jk},m_k^{\jj{y},jk})
+ \mathbf{E}[\dH_k, Y_j^{\jj{x},\jj{y}}] \big).
\end{equation}
Given $0\leq j < n$, let
\begin{align*}
Q_j &:=
\sum_{(\jj{x},\jj{y},\jj{z}) \in (S^{jk})^3}
\zeta(\jj{x},\jj{y},\jj{z}) 
\DTV(\mu_k^{\jj{y},jk},m_k^{\jj{y},jk}), \\
R_j &:=
\sum_{(\jj{x},\jj{y},\jj{z}) \in (S^{jk})^3}
\zeta(\jj{x},\jj{y},\jj{z}) 
\mathbf{E}[\dH_k, Y_j^{\jj{x},\jj{y}}],
\end{align*}
so that \eqref{eqn:Delta-1} can be written as
\begin{equation}\label{eqn:Delta-2}
\Delta \leq 
Q + R,
\quad\text{where}\quad
Q := \frac 1n \sum_{j=0}^{n-1} Q_j
\quad\text{and}\quad
R := \frac 1n \sum_{j=0}^{n-1} R_j.
\end{equation}
We will obtain a bound on $R$ in terms of $Q$ by relating $R_j$ to $Q_{j-1}$ and $R_{j-1}$. The following lemma is our crucial estimate; it relies on properties of the coupling $Y_j^{\jj{x},\jj{y}}$ that used the Markov property of $m$, and also relies on the mixing property of $m$.

\begin{lem}\label{lem:Eleq}
For all $(\jm{x},\jm{y},\jm{z}) \in (S^{(j-1)k})^3$ and $(u,v,w) \in (S^k)^3$, we have
\begin{equation}\label{eqn:Efg}
\mathbf{E}[\dH_k, Y_j^{\jm{x} u, \jm{y} v} ]
\leq f(u,v,w) + (1-\lambda) g(u,v,w),
\end{equation}
where $\lambda = q/k$, with $q>0$ and $k\in \mathbb{N}$ as in the statement of Theorem \ref{thm: d_bar distance and KL-divergence}.
\end{lem}
\begin{proof}
There are three cases to consider.
\begin{itemize}
\item If $v \neq w$, then $f(u, v, w) = 1$, and the result follows since $\dH_k \leq 1$.
\item If $v = w$ and $u_k = w_k$, then
$Y_j^{\jm{x} u, \jm{y} v}$ is an optimal coupling of a measure with itself, and thus is supported on the diagonal in $(S^k)^2$, so
$\mathbf{E}[\dH_k, Y_j^{\jm{x} u, \jm{y} v} ] = 0$, and the result follows.
\item If $v = w$ and $u_k \neq w_k$, then by \eqref{eqn:stay-coupled}, we have $\dH_k(u, w) = 1$, so $g(u, v, w) = 1$, and the right-hand side of \eqref{eqn:Efg} is equal to $1-\frac q k$.
Since the transition matrix $T$ satisfies $(T^k)_{ab} \geq q$ for all $a,b\in S$, it follows that $\mathbf{E}[\dH_k, Y_j^{\jm{x} u, \jm{y} v} ] \leq 1-\frac q k$, and we are done.\qedhere
\end{itemize}
\end{proof}

Reindexing $\sum_{(\jj{x},\jj{y},\jj{z})}$ as $\sum_{(\jm{x},\jm{y},\jm{z})} \sum_{(u,v,w)}$, with the latter sums taken over $(S^{(j-1)k})^3$ and $(S^k)^3$, 
we can use Lemma \ref{lem:Eleq} together with \eqref{eqn:Ef} and \eqref{eqn:Eg} to get
\begin{multline*}
\sum_{(\jj{x},\jj{y},\jj{z}) \in (S^{jk})^3} \zeta(\jj{x},\jj{y},\jj{z}) \mathbf{E}[\dH_k, Y_j^{\jj{x},\jj{y}}]
\leq
\sum_{(\jm{x},\jm{y},\jm{z})\in (S^{(j-1)k})^3}
\zeta(\jm{x},\jm{y},\jm{z})
\mathbf{E}\big[
f + (1-\lambda)g,
Z_{j-1}^{\jm{x},\jm{y},\jm{z}}\big] \\
\leq
\sum_{(\jm{x},\jm{y},\jm{z})\in (S^{(j-1)k})^3}
\zeta(\jm{x},\jm{y},\jm{z})\Big(
\DTV(\mu_k^{\jm{y},(j-1)k},m_k^{\jm{y},(j-1)k})
+ (1-\lambda) 
\mathbf{E}[\dH_k, Y_{j-1}^{\jm{x},\jm{y}}] \Big).
\end{multline*}
Using the $Q_j, R_j$ notation and summing over $j$, this gives
\[
R_j \leq Q_{j-1} + (1-\lambda)R_{j-1}
\quad\Rightarrow\quad
R \leq Q + (1-\lambda)R,
\]
from which we deduce that $0 \leq Q - \lambda R$, so $R \leq \lambda^{-1} Q$, which we use in \eqref{eqn:Delta-2} to obtain
\[
\Delta \leq Q + \lambda^{-1} Q
= (1 + \lambda^{-1}) \frac 1n \sum_{j=0}^{n-1} Q_j.
\]
With $K := 1+\lambda^{-1} = 1 + \frac kq$, we can rewrite this as
\begin{equation}\label{eqn:Delta-last}
\Delta \leq \frac Kn \sum_{j=0}^{n-1}
\sum_{(\jj{x},\jj{y},\jj{z}) \in (S^{jk})^3}
\zeta(\jj{x},\jj{y},\jj{z}) 
\DTV(\mu_k^{y,jk},m_k^{y,jk}).
\end{equation}
Since $\pi^2_* \zeta = \mu_{nk}$, 
we have $\sum_{(\jj{x},\jj{y},\jj{z}) \in (S^{jk})^3} \zeta(x,y,z) = \sum_{\jj{y}\in S^{jk}}\mu_{jk}(y)$, so \eqref{eqn:Delta-last} implies \eqref{eqn:dTV} and completes the proof of Proposition \ref{prop:dTV}.

\subsection{From total variation to a $\bar{d}$-estimate}\label{sec:back-to-d-bar}

Combining Proposition \ref{prop:dTV} and Lemma \ref{lem:Pinsker}, and using concavity of the square root function, we get
\begin{align*}
\bar{d}_{nk}(m_{nk},\mu_{nk})
&\leq \frac Kn \sum_{j=0}^{n-1} \sum_{\jj{y}\in S^{jk}} \mu_{jk}(\jj{y})
\sqrt{\frac 12 D(\mu_k^{\jj{y},jk} \| m_k^{\jj{y},jk})} \\
&\leq \frac K{\sqrt{2}}
\bigg( \frac 1n \sum_{j=0}^{n-1} \sum_{\jj{y}\in S^{jk}} \mu_{jk}(y) D(\mu_k^{\jj{y},jk} \| m_k^{\jj{y},jk}) \bigg)^{1/2}.
\end{align*}
Defining $f_j \colon S^{jk} \to \RR$ by $f_j(\jj{y}) = D(\mu_k^{\jj{y},jk} \| m_k^{\jj{y},jk})$, we can write this as
\begin{equation}\label{eqn:av-E}
\bar{d}_{nk}(m_{nk},\mu_{nk})
\leq \frac K{\sqrt{2}} \Big( \frac 1n \sum_{j=0}^{n-1} \EE[ f_j, \mu_{jk} ] \Big)^{1/2}.
\end{equation}
For every $j$, every $\jj{y}\in S^{jk}$, and every $w\in S^k$, we have
\[
\mu_{(j+1)k}(\jj{y}w) = \mu_{jk}(y) \mu_k^{\jj{y},jk}(w).
\]
Using this, we see that
\[
\log \Big( \frac{\mu_k^{\jj{y},jk}(w)}{m_k^{\jj{y},jk}(w)} \Big)
= \log \Big( \frac{\mu_{(j+1)k}(\jj{y} w)}{m_{(j+1)k}(\jj{y} w)} \Big)
- \log \Big( \frac{\mu_{jk}(\jj{y})}{m_{jk}(\jj{y})} \Big),
\]
and consequently,
\begin{align*}
\mu_{jk}(\jj{y}) f_j(\jj{y})
&= \sum_{w\in S^k} \mu_{jk}(\jj{y})
\mu_k^{\jj{y},jk}(w) \log \Big(
\frac{\mu_k^{\jj{y},jk}(w)}{m_k^{\jj{y},jk}(w)} \Big)\\
&= \left(\sum_{w\in S^k} \mu_{(j+1)k}(yw)
\log \Big( \frac{\mu_{(j+1)k}(\jj{y} w)}{m_{(j+1)k}(\jj{y} w)} \Big)
\right)
-  \mu_{jk}(\jj{y}) \log \Big( \frac{\mu_{jk}(\jj{y})}{m_{jk}(\jj{y})} \Big).
\end{align*}
Summing over $\jj{y} \in S^{jk}$, we see that $\sum_{\jj{y}} \sum_w$ can be replaced by $\sum_{\jp{y} \in S^{(j+1)k}}$, and thus
\[
\EE[f_j, \mu_{jk}]
= D(\mu_{(j+1)k} \| m_{(j+1)k}) - D(\mu_{jk} \| m_{jk}).
\]
Using this in \eqref{eqn:av-E}, the sum becomes telescoping, and since $\mu_{jk} = m_{jk}$ when $j=0$, we obtain
\[
\sum_{j=0}^{n-1} \EE[f_j, \mu_{jk}] = D(\mu_{nk} \| m_{nk}).
\]
Thus \eqref{eqn:av-E} gives
\begin{equation}\label{eqn:Ksqrt2}
\bar{d}_{nk}(m_{nk},\mu_{nk})
\leq \frac{K}{\sqrt{2}} \Big( \frac 1n D(\mu_{nk} \| m_{nk}) \Big)^{1/2}.
\end{equation}
Since $k\geq 1 \geq q > 0$, we have $K = 1 + \frac kq \leq \frac{2k}q$, so $\frac{K}{\sqrt{2}} \leq \frac kq \sqrt{2}$, and thus \eqref{eqn:Ksqrt2} implies \eqref{eqn:Marton}, completing the proof of Theorem \ref{thm: d_bar distance and KL-divergence}.

\section{Estimates on the pressure function}\label{section: Pressure}

Now we use Theorems \ref{thm: Main Theorem} and \ref{thm: dual form of d_bar} to prove Theorem \ref{thm: pressure estimates}.
Fix $\phi \in \OmH$ and $\mu\in \MMM_\sigma(\Sigma)$.
For any $a\in \mathbb{R}$, consider the set of measures
\[
M^\phi_a := \Big\{ \mu \in \MMM_\sigma(\Sigma) : \int_\Sigma \phi \,d\mu = a \Big\},
\]
and let $I_\phi := \{ a\in \RR : M_a^\phi \neq \emptyset \}$.
Given $a\in I_\phi$, let
\[
H(a) := \sup \{ h(\mu) : \mu \in M^\phi_a \}.
\] 
Let $\bar{a} := \int \phi \,dm$, where $m$ is the unique MME.
Observe that $H(\bar{a}) = h(m) = \htop(\Sigma)$.
For each $t\in \mathbb{R}$, the variational principle for topological pressure gives
\begin{equation}\label{eqn:Pt}
\begin{aligned}
P(t\phi)&= \sup \left\{ h(\mu)+t\int \phi\ d\mu : \mu \in \MMM_\sigma(\Sigma)\right\}\\
&= \sup_{a\in I_\phi} \sup \{ h(\mu) + ta : \mu \in M_a^\phi \}
= \sup_{a\in I_\phi}\left\{H(a)+ ta \right\}.
\end{aligned}
\end{equation}
Given $a\in I_\phi$ and $\mu \in M^\phi_a$, we have
\[
|a - \bar{a}| = |\mu(\phi) - m(\phi)|
\leq \bar{d}(\mu,m)
\leq C\sqrt{\htop(\Sigma) - h(\mu)},
\]
where the first inequality uses Theorem \ref{thm: dual form of d_bar}, and the second uses Theorem \ref{thm: Main Theorem}. Squaring both sides gives
\[
(a-\bar{a})^2 \leq C^2\big(\htop(\Sigma) - h(\mu)\big),
\]
from which we deduce that
\[
h(\mu) \leq \htop(\Sigma) - C^{-2}(a-\bar{a})^2.
\]
Taking a supremum over all $\mu \in M^\phi_a$ gives
\[
H(a) \leq \htop(\Sigma) - C^{-2}(a-\bar{a})^2,
\]
and combining this with \eqref{eqn:Pt}, we get
\[
P(t\phi) \leq \sup_{a\in I_\phi} \big( \htop(\Sigma) - C^{-2}(a-\bar{a})^2 + ta \big).
\]
Writing $s := a - \bar{a} = a - m(\phi)$, this yields
\[
P(t\phi) \leq \sup_{s\in \RR} \big( \htop(\Sigma) - C^{-2} s^2 + ts + tm(\phi) \big).
\]
The quantity inside the supremum takes its maximum value when $s = \frac 12 tC^2$, so
\[
P(t\phi) \leq \htop(\Sigma) - \frac 1{C^{2}}  \frac 14 t^2 C^4
+ \frac 12 t^2 C^2 + tm(\phi)
= \htop(\Sigma) + m(\phi) t + \frac 14C^2 t^2,
\]
which proves \eqref{eqn:P-leq} with $M = \frac 14 C^2$. This completes the proof of Theorem \ref{thm: pressure estimates}.

\bibliographystyle{alpha}
\bibliography{references}

@book {walters2000introduction,
    AUTHOR = {Walters, Peter},
     TITLE = {An introduction to ergodic theory},
    SERIES = {Graduate Texts in Mathematics},
    VOLUME = {79},
 PUBLISHER = {Springer-Verlag, New York-Berlin},
      YEAR = {1982},
     PAGES = {ix+250},
      ISBN = {0-387-90599-5},
   MRCLASS = {28Dxx (54H20 58F11)},
  MRNUMBER = {648108},
MRREVIEWER = {M.\ A.\ Akcoglu},
}

@book {Villani2008OptimalTO,
    AUTHOR = {Villani, C\'edric},
     TITLE = {Optimal transport: Old and new},
    SERIES = {Grundlehren der mathematischen Wissenschaften [Fundamental
              Principles of Mathematical Sciences]},
    VOLUME = {338},
 PUBLISHER = {Springer-Verlag, Berlin},
      YEAR = {2009},
     PAGES = {xxii+973},
      ISBN = {978-3-540-71049-3},
   MRCLASS = {49-02 (28A75 37J50 49Q20 53C23 58E30)},
  MRNUMBER = {2459454},
MRREVIEWER = {Dario\ Cordero-Erausquin},
       DOI = {10.1007/978-3-540-71050-9},
       URL = {https://doi.org/10.1007/978-3-540-71050-9},
}

@book{glasner2003ergodic,
    AUTHOR = {Glasner, Eli},
     TITLE = {Ergodic theory via joinings},
    SERIES = {Mathematical Surveys and Monographs},
    VOLUME = {101},
 PUBLISHER = {American Mathematical Society, Providence, RI},
      YEAR = {2003},
     PAGES = {xii+384},
      ISBN = {0-8218-3372-3},
   MRCLASS = {37A15 (28Dxx 37A25 37A35 37A45 37B99 54H20)},
  MRNUMBER = {1958753},
MRREVIEWER = {Andr\'es\ del\ Junco},
       DOI = {10.1090/surv/101},
       URL = {https://doi.org/10.1090/surv/101},
}

@article{coelho1998criteria,
    AUTHOR = {Coelho, Zaqueu and Quas, Anthony N.},
     TITLE = {Criteria for {$\overline d$}-continuity},
   JOURNAL = {Trans. Amer. Math. Soc.},
  FJOURNAL = {Transactions of the American Mathematical Society},
    VOLUME = {350},
      YEAR = {1998},
    NUMBER = {8},
     PAGES = {3257--3268},
      ISSN = {0002-9947,1088-6850},
   MRCLASS = {28D05},
  MRNUMBER = {1422894},
MRREVIEWER = {Mariusz\ Lema\'nczyk},
       DOI = {10.1090/S0002-9947-98-01923-0},
       URL = {https://doi.org/10.1090/S0002-9947-98-01923-0},
}

@article{bhullar2026continuity,
    AUTHOR = {Bhullar, Jasmine},
     TITLE = {{$\overline {d}$}-continuity for countable state shifts},
   JOURNAL = {Ergodic Theory Dynam. Systems},
  FJOURNAL = {Ergodic Theory and Dynamical Systems},
    VOLUME = {46},
      YEAR = {2026},
    NUMBER = {2},
     PAGES = {466--489},
      ISSN = {0143-3857,1469-4417},
   MRCLASS = {37D35 (37A30 37B10)},
  MRNUMBER = {5012852},
MRREVIEWER = {Anthony\ Quas},
       DOI = {10.1017/etds.2025.10241},
       URL = {https://doi.org/10.1017/etds.2025.10241},
}

@article {rR21,
    AUTHOR = {R{\"u}hr, Ren\'e},
     TITLE = {Pressure inequalities for {G}ibbs measures of countable
              {M}arkov shifts},
   JOURNAL = {Dyn. Syst.},
  FJOURNAL = {Dynamical Systems. An International Journal},
    VOLUME = {36},
      YEAR = {2021},
    NUMBER = {2},
     PAGES = {332--339},
      ISSN = {1468-9367,1468-9375},
   MRCLASS = {37D35},
  MRNUMBER = {4265451},
MRREVIEWER = {Cao\ Zhao},
       DOI = {10.1080/14689367.2021.1905777},
       URL = {https://doi-org.ezproxy.lib.uh.edu/10.1080/14689367.2021.1905777},
}

@article {rR16,
    AUTHOR = {R{\"u}hr, Rene},
     TITLE = {Effectivity of uniqueness of the maximal entropy measure on
              {$p$}-adic homogeneous spaces},
   JOURNAL = {Ergodic Theory Dynam. Systems},
  FJOURNAL = {Ergodic Theory and Dynamical Systems},
    VOLUME = {36},
      YEAR = {2016},
    NUMBER = {6},
     PAGES = {1972--1988},
      ISSN = {0143-3857,1469-4417},
   MRCLASS = {37P55 (22E35 37A35)},
  MRNUMBER = {3530473},
MRREVIEWER = {Thomas\ Ward},
       DOI = {10.1017/etds.2014.148},
       URL = {https://doi-org.ezproxy.lib.uh.edu/10.1017/etds.2014.148},
}

@article {iK17,
    AUTHOR = {Khayutin, Ilya},
     TITLE = {Large deviations and effective equidistribution},
   JOURNAL = {Int. Math. Res. Not. IMRN},
  FJOURNAL = {International Mathematics Research Notices. IMRN},
      YEAR = {2017},
    NUMBER = {10},
     PAGES = {3050--3106},
      ISSN = {1073-7928,1687-0247},
   MRCLASS = {37P55 (37A30 37A50)},
  MRNUMBER = {3658132},
MRREVIEWER = {Thomas\ Ward},
       DOI = {10.1093/imrn/rnw099},
       URL = {https://doi-org.ezproxy.lib.uh.edu/10.1093/imrn/rnw099},
}

@article{ruhr2022effective,
    AUTHOR = {R\"uhr, Ren\'e{} and Sarig, Omri},
     TITLE = {Effective intrinsic ergodicity for countable state {M}arkov
              shifts},
   JOURNAL = {Israel J. Math.},
  FJOURNAL = {Israel Journal of Mathematics},
    VOLUME = {251},
      YEAR = {2022},
    NUMBER = {2},
     PAGES = {679--735},
      ISSN = {0021-2172,1565-8511},
   MRCLASS = {37B10},
  MRNUMBER = {4527555},
MRREVIEWER = {Anthony\ Quas},
       DOI = {10.1007/s11856-022-2436-x},
       URL = {https://doi.org/10.1007/s11856-022-2436-x},
}

@article {sK17,
    AUTHOR = {Kadyrov, Shirali},
     TITLE = {Effective equidistribution of periodic orbits for subshifts of
              finite type},
   JOURNAL = {Colloq. Math.},
  FJOURNAL = {Colloquium Mathematicum},
    VOLUME = {149},
      YEAR = {2017},
    NUMBER = {1},
     PAGES = {93--101},
      ISSN = {0010-1354,1730-6302},
   MRCLASS = {37A35 (28D20 37D05 37D20)},
  MRNUMBER = {3684406},
MRREVIEWER = {T\'ulio\ O.\ Carvalho},
       DOI = {10.4064/cm6653-9-2016},
       URL = {https://doi-org.ezproxy.lib.uh.edu/10.4064/cm6653-9-2016},
}

@article{kadyrov2015effective,
    AUTHOR = {Kadyrov, Shirali},
     TITLE = {Effective uniqueness of {P}arry measure and exceptional sets
              in ergodic theory},
   JOURNAL = {Monatsh. Math.},
  FJOURNAL = {Monatshefte f\"ur Mathematik},
    VOLUME = {178},
      YEAR = {2015},
    NUMBER = {2},
     PAGES = {237--249},
      ISSN = {0026-9255,1436-5081},
   MRCLASS = {37A35 (28D20 37C45)},
  MRNUMBER = {3394424},
MRREVIEWER = {Yun\ Zhao},
       DOI = {10.1007/s00605-014-0690-7},
       URL = {https://doi.org/10.1007/s00605-014-0690-7},
}

@article{parry1964intrinsic,
    AUTHOR = {Parry, William},
     TITLE = {Intrinsic {M}arkov chains},
   JOURNAL = {Trans. Amer. Math. Soc.},
  FJOURNAL = {Transactions of the American Mathematical Society},
    VOLUME = {112},
      YEAR = {1964},
     PAGES = {55--66},
      ISSN = {0002-9947,1088-6850},
   MRCLASS = {60.65},
  MRNUMBER = {161372},
MRREVIEWER = {H.\ P.\ Edmundson},
       DOI = {10.2307/1994009},
       URL = {https://doi.org/10.2307/1994009},
}

@phdthesis{polo2011equidistribution,
    AUTHOR = {Polo, Fabrizio},
     TITLE = {Equidistribution in chaotic dynamical systems},
      NOTE = {Thesis (Ph.D.)--The Ohio State University},
 PUBLISHER = {ProQuest LLC, Ann Arbor, MI},
      YEAR = {2011},
     PAGES = {111},
      ISBN = {978-1124-91043-7},
   MRCLASS = {99-05},
  MRNUMBER = {2942253},
       URL =
              {http://gateway.proquest.com/openurl?url_ver=Z39.88-2004&rft_val_fmt=info:ofi/fmt:kev:mtx:dissertation&res_dat=xri:pqdiss&rft_dat=xri:pqdiss:3477026},
}

@article{marton1998measure,
    AUTHOR = {Marton, Katalin},
     TITLE = {Measure concentration for a class of random processes},
   JOURNAL = {Probab. Theory Related Fields},
  FJOURNAL = {Probability Theory and Related Fields},
    VOLUME = {110},
      YEAR = {1998},
    NUMBER = {3},
     PAGES = {427--439},
      ISSN = {0178-8051,1432-2064},
   MRCLASS = {60G10 (28A99)},
  MRNUMBER = {1616492},
       DOI = {10.1007/s004400050154},
       URL = {https://doi.org/10.1007/s004400050154},
}

@article{marton1996bounding,
    AUTHOR = {Marton, K.},
     TITLE = {Bounding {$\overline d$}-distance by informational divergence:
              a method to prove measure concentration},
   JOURNAL = {Ann. Probab.},
  FJOURNAL = {The Annals of Probability},
    VOLUME = {24},
      YEAR = {1996},
    NUMBER = {2},
     PAGES = {857--866},
      ISSN = {0091-1798,2168-894X},
   MRCLASS = {60F10 (60G70)},
  MRNUMBER = {1404531},
       DOI = {10.1214/aop/1039639365},
       URL = {https://doi.org/10.1214/aop/1039639365},
}

@article{kucherenko2023asymptotic,
  title={Asymptotic behavior of the pressure function for Hölder potentials},
  author={Kucherenko, Tamara and Quas, Anthony},
  journal={arXiv preprint arXiv:2302.14839},
  year={2023}
}

@article {KQ22,
    AUTHOR = {Kucherenko, Tamara and Quas, Anthony},
     TITLE = {Flexibility of the pressure function},
   JOURNAL = {Comm. Math. Phys.},
  FJOURNAL = {Communications in Mathematical Physics},
    VOLUME = {395},
      YEAR = {2022},
    NUMBER = {3},
     PAGES = {1431--1461},
      ISSN = {0010-3616,1432-0916},
   MRCLASS = {37D35},
  MRNUMBER = {4496390},
MRREVIEWER = {Victor\ Vargas},
       DOI = {10.1007/s00220-022-04466-y},
       URL = {https://doi-org.ezproxy.lib.uh.edu/10.1007/s00220-022-04466-y},
}

@article {MP24,
    AUTHOR = {Ma, Liangang and Pollicott, Mark},
     TITLE = {Rigidity of pressures of {H}\"older potentials and the fitting
              of analytic functions through them},
   JOURNAL = {Ergodic Theory Dynam. Systems},
  FJOURNAL = {Ergodic Theory and Dynamical Systems},
    VOLUME = {44},
      YEAR = {2024},
    NUMBER = {12},
     PAGES = {3530--3564},
      ISSN = {0143-3857,1469-4417},
   MRCLASS = {37D35 (37A50)},
  MRNUMBER = {4818935},
MRREVIEWER = {Seyed\ Mohsen\ Moosavi},
       DOI = {10.1017/etds.2024.9},
       URL = {https://doi-org.ezproxy.lib.uh.edu/10.1017/etds.2024.9},
}

@incollection {dO70,
    AUTHOR = {Ornstein, D. S.},
     TITLE = {Imbedding {B}ernoulli shifts in flows},
 BOOKTITLE = {Contributions to {E}rgodic {T}heory and {P}robability ({P}roc.
              {C}onf., {O}hio {S}tate {U}niv., {C}olumbus, {O}hio, 1970)},
    SERIES = {Lecture Notes in Math.},
    VOLUME = {Vol. 160},
     PAGES = {178--218},
 PUBLISHER = {Springer, Berlin-New York},
      YEAR = {1970},
   MRCLASS = {28.70},
  MRNUMBER = {272985},
MRREVIEWER = {R.\ L.\ Adler},
}

@article {OW74,
    AUTHOR = {Ornstein, Donald S. and Weiss, Benjamin},
     TITLE = {Finitely determined implies very weak {B}ernoulli},
   JOURNAL = {Israel J. Math.},
  FJOURNAL = {Israel Journal of Mathematics},
    VOLUME = {17},
      YEAR = {1974},
     PAGES = {94--104},
      ISSN = {0021-2172},
   MRCLASS = {28A65},
  MRNUMBER = {346132},
MRREVIEWER = {Nathaniel\ Friedman},
       DOI = {10.1007/BF02756830},
       URL = {https://doi-org.ezproxy.lib.uh.edu/10.1007/BF02756830},
}

@article {dO70a,
    AUTHOR = {Ornstein, Donald},
     TITLE = {Bernoulli shifts with the same entropy are isomorphic},
   JOURNAL = {Advances in Math.},
  FJOURNAL = {Advances in Mathematics},
    VOLUME = {4},
      YEAR = {1970},
     PAGES = {337--352},
      ISSN = {0001-8708},
   MRCLASS = {28.70},
  MRNUMBER = {257322},
MRREVIEWER = {U.\ Krengel},
       DOI = {10.1016/0001-8708(70)90029-0},
       URL = {https://doi.org/10.1016/0001-8708(70)90029-0},
}

@book {dO74,
    AUTHOR = {Ornstein, Donald S.},
     TITLE = {Ergodic theory, randomness, and dynamical systems},
    SERIES = {Yale Mathematical Monographs},
    VOLUME = {No. 5},
      NOTE = {James K. Whittemore Lectures in Mathematics given at Yale
              University},
 PUBLISHER = {Yale University Press, New Haven, Conn.-London},
      YEAR = {1974},
     PAGES = {vii+141},
   MRCLASS = {28A65},
  MRNUMBER = {447525},
MRREVIEWER = {Paul\ C.\ Shields},
}

@article {AW67,
    AUTHOR = {Adler, R. L. and Weiss, B.},
     TITLE = {Entropy, a complete metric invariant for automorphisms of the
              torus},
   JOURNAL = {Proc. Nat. Acad. Sci. U.S.A.},
  FJOURNAL = {Proceedings of the National Academy of Sciences of the United
              States of America},
    VOLUME = {57},
      YEAR = {1967},
     PAGES = {1573--1576},
      ISSN = {0027-8424},
   MRCLASS = {28.70},
  MRNUMBER = {212156},
MRREVIEWER = {D.\ Newton},
       DOI = {10.1073/pnas.57.6.1573},
       URL = {https://doi.org/10.1073/pnas.57.6.1573},
}

@book {AW70,
    AUTHOR = {Adler, Roy L. and Weiss, Benjamin},
     TITLE = {Similarity of automorphisms of the torus},
    SERIES = {Memoirs of the American Mathematical Society},
    VOLUME = {No. 98},
 PUBLISHER = {American Mathematical Society, Providence, RI},
      YEAR = {1970},
     PAGES = {ii+43},
   MRCLASS = {28.70},
  MRNUMBER = {257315},
MRREVIEWER = {D.\ Newton},
}

@article {bW73,
    AUTHOR = {Weiss, Benjamin},
     TITLE = {Subshifts of finite type and sofic systems},
   JOURNAL = {Monatsh. Math.},
  FJOURNAL = {Monatshefte f\"ur Mathematik},
    VOLUME = {77},
      YEAR = {1973},
     PAGES = {462--474},
      ISSN = {0026-9255,1436-5081},
   MRCLASS = {28A65 (54H20)},
  MRNUMBER = {340556},
MRREVIEWER = {R.\ L.\ Adler},
       DOI = {10.1007/BF01295322},
       URL = {https://doi.org/10.1007/BF01295322},
}

@incollection {CT21,
    AUTHOR = {Climenhaga, Vaughn and Thompson, Daniel J.},
     TITLE = {Beyond {B}owen's specification property},
 BOOKTITLE = {Thermodynamic formalism},
    SERIES = {Lecture Notes in Math.},
    VOLUME = {2290},
     PAGES = {3--82},
 PUBLISHER = {Springer, Cham},
      YEAR = {[2021] \copyright 2021},
      ISBN = {978-3-030-74862-3; 978-3-030-74863-0},
   MRCLASS = {37D35},
  MRNUMBER = {4436821},
       DOI = {10.1007/978-3-030-74863-0\_1},
       URL = {https://doi.org/10.1007/978-3-030-74863-0_1},
}

@article {dR73,
    AUTHOR = {Ruelle, David},
     TITLE = {Statistical mechanics on a compact set with {$Z\sp{v}$} action
              satisfying expansiveness and specification},
   JOURNAL = {Trans. Amer. Math. Soc.},
  FJOURNAL = {Transactions of the American Mathematical Society},
    VOLUME = {187},
      YEAR = {1973},
     PAGES = {237--251},
      ISSN = {0002-9947,1088-6850},
   MRCLASS = {28A65 (54H20 58F99 82.28)},
  MRNUMBER = {417391},
MRREVIEWER = {R.\ L.\ Adler},
       DOI = {10.2307/1996437},
       URL = {https://doi.org/10.2307/1996437},
}

\end{document}